\documentclass[a4paper,12pt,reqno]{amsart}
\usepackage{geometry}
\usepackage[utf8]{inputenc}
\usepackage{upgreek}
\usepackage{amsmath, mathtools, esint, dsfont, amssymb, enumitem}
\usepackage[natbib=true, style=numeric, sorting=nyvt]{biblatex}
\usepackage[colorlinks, allcolors=black]{hyperref}
\newtheorem{theorem}{Theorem}
\newtheorem*{xremark}{Remark}

\newtheorem{proposition}{Proposition}
\newtheorem{lemma}{Lemma}

\newcommand{\N}{\mathbb{N}}
\newcommand{\R}{\mathbb{R}}
\newcommand{\E}{\mathbb{E}}
\newcommand{\p}{\mathbb{P}}

\newcommand{\C}{\mathbb{C}}

\begin{document}
\title{Character sums over smooth numbers}
\author{Seth Hardy}
\address{Mathematics Institute, Zeeman Building, University of Warwick, Coventry CV4 7AL, England}
\email{Seth.Hardy@warwick.ac.uk}
\author{Max Wenqiang Xu} 
\address{Yau Mathematical Sciences Center, Tsinghua University, Beijing, China}
\email{maxxu1729@gmail.com}
\date{\today} 
\begin{abstract}
Let $\Psi (x,y)$ denote the count of $y$-smooth numbers below $x$ and $P(n)$ denote the largest prime factor of $n$. We show that
\[
\frac{1}{\varphi(q)} \sum_{\chi \bmod q} \Bigl| \sum_{\substack{n \leq x \\ P(n) \leq y}} \chi(n) \Bigr| = o \Bigl( \sqrt{\Psi(x,y)} \Bigr),
\]
whenever $(\log x)^6 \leq y \leq x^{\frac{1}{32 \log \log x}}$ and $q \geq x^{1 + \varepsilon}$ for some small quantifiable $\varepsilon > 0$. The saving is substantial when $\varepsilon$ is fixed away from zero, and we prove similar results for continuous characters and completely multiplicative twists of these sums.
\end{abstract}
\maketitle

\section{Introduction}

\subsection{Main result}
Thanks to developments in the theory of \emph{random multiplicative functions}, our understanding of the typical behaviour of partial sums of Dirichlet characters and continuous characters has advanced substantially in recent years. The Steinhaus random multiplicative function $f$ is defined by letting $(f(p))_{p \text{ prime}}$ be independent random variables uniformly distributed on the complex unit circle, and setting
\[
f(n) \coloneqq \prod_{i=1}^r f(p_i)^{\alpha_i},
\]
whenever $n = p_1^{\alpha_1} \dots p_r^{\alpha_r}$. A key property of the Steinhaus random multiplicative function is that it satisfies the orthogonality relation, $\E \bigl[ f(n) \overline{f(m)} \bigr] = \mathbf{1}_{n=m}$, which mimics orthogonality of characters. For example, it is quite straightforward to prove~\cite[Lemma 2.1]{GranvilleSound} that for $x^k \leq q$, we have
\[
\frac{1}{\varphi(q)} \sum_{\chi \bmod q} \Bigl| \sum_{n \leq x} \chi (n) \Bigr|^{2k} = \E \Bigl| \sum_{\substack{n \leq x \\ (n,q)=1}} f(n) \Bigr|^{2k},
\]
a result that was leveraged in the seminal work of~\citet{GranvilleSound} to investigate large values of partial sums of Dirichlet characters.

Relevant to our work, in a recent preprint of~\citet[Theorem~1]{HarperTypical} it was shown that, for $q$ prime, we have
\[
\frac{1}{q-1} \sum_{\chi \bmod q} \Bigl| \sum_{n \leq x} \chi (n) \Bigr| = o (\sqrt{x} ),
\]
so long as $x, q/x \rightarrow \infty$, which improves the classical square-root bound that can be obtained from applying the Cauchy--Schwarz inequality and orthogonality. The result is quantitative, saving at most a factor of $(\log \log x)^{1/4}$, which is expected to be optimal in light of Harper's~\cite[Theorem~1]{HarperLM} previous result that
\[
\E \Bigl| \sum_{n \leq x} f(n) \Bigr| \asymp \frac{\sqrt{x}}{(\log \log x)^{1/4}} ,
\]
for $f$ a Steinhaus random multiplicative function. In fact, the proof of the character sum result works by explicitly connecting the problem to (a certain stage in the proof of) the random multiplicative case. 

In this paper, we approach the problem of bounding the first absolute moment of Dirichlet characters and continuous characters over $y$-smooth numbers, that is, numbers whose prime factors are all less than or equal to $y$; see~\cite{GranvilleSmooths} for a survey on smooth numbers. Let us briefly discuss what is known in the random multiplicative setting. Letting $P(n)$ denote the largest prime factor of $n$, the authors have previously proved~\cite[Theorem~1.1]{HardyXu} that
\[
\E \Bigl| \sum_{\substack{n \leq x \\ P(n) \leq y}} f(n) \Bigr| = o \Bigl( \sqrt{\Psi(x,y)} \Bigr),
\]
uniformly $2 \leq y \leq x$, where $\Psi(x,y) \coloneqq \# \{ n \leq x : P(n) \leq y\}$ is the count of $y$-smooth numbers below $x$, thus beating the square-root bound uniformly on the entire range of smoothness parameter. Moreover, in the range $(\log x)^{1 + \varepsilon} \leq y \leq x^{1/(\log \log x)^{1 + \varepsilon}}$, it was shown~\cite[Theorem~1.2]{HardyXu} that
\[
\E \Bigl| \sum_{\substack{n \leq x \\ P(n) \leq y}} f(n) \Bigr| \leq \sqrt{\Psi(x,y) \exp \bigl(- u \log 2 (1 + o_{\varepsilon} (1)) \bigr)},
\]
where $u = \frac{\log x}{\log y}$. The saving is large because the expectation is dominated by rare events, see~\cite[Section~1.3]{HardyXu}. In the present paper, we adapt this argument to the deterministic setting of Dirichlet characters, continuous characters, and completely multiplicative twists of these. Recall that for any given $2 \leq y \leq x$, the saddle point $\alpha(x,y)$ is defined as the unique solution to
\begin{equation}\label{equ: saddle defn}
\sum_{p \leq y} \frac{\log p}{p^{\alpha(x,y)} - 1} = \log x,
\end{equation}
see~\cite[Equation~(2.2)]{HildTen86}. Importantly, in the range $(\log x)^2 \leq y \leq x^{1/3}$, say, it can be shown~\cite[Lemmas~2 and~3]{HildTen86} that
\begin{equation}\label{equ: saddle point est}
\alpha (x,y) = 1 - \frac{\log (u \log u)}{\log y} + O \biggl( \frac{1}{\log y} \biggr).
\end{equation}
For the family of Dirichlet characters, our main result is the following:
\begin{theorem}\label{t: main}
Suppose that $2 \leq x \leq q$ and $(\log x)^6 \leq y \leq x^{\frac{1}{32 \log \log x}}$. Let $\alpha (x,y)$ be the corresponding saddle point and let $u = \frac{\log x}{\log y}$. Let $\sum_{\chi \bmod q}$ denote the sum over all Dirichlet characters $\chi \bmod q$, and define
\[
S(x,y,q) \coloneqq (\log x)^{3/2} (\log y)^{1/2} \biggl( y^{1/2 - \alpha(x,y)} +  \exp \biggl\{ - \frac{u}{4} \biggl( \frac{\log \min\{ q, x^{3/2} \}}{\log x} - 1 \biggr)^2  \biggr\} \biggr).
\]
Then we have
\[
\frac{1}{\varphi(q)} \sum_{\chi \bmod q} \Bigl| \sum_{\substack{n \leq x \\ P(n) \leq y}} \chi(n) \Bigr| \ll \sqrt{\Psi(x,y)S(x,y,q)},
\]
where the implied constant is uniform in all variables. For any function $\eta(x) \rightarrow \infty$, the saving factor $S(x,y,q)$ is $o_{\eta} (1)$ uniformly on the range $(\log x)^{6} \leq y \leq x^{\frac{1}{32 \log \log x}}$ and $ \frac{\log q}{\log x} \geq 1 + \sqrt{\frac{6 \log \log x + 2 \log \log y + \eta (x)}{u}}$.
\end{theorem}

Let us consider the size of this saving in different ranges of $y$: first, if $q = x^{1 + \varepsilon}$ for fixed $\varepsilon \in (0,1/2)$, then on the range $\log x  = o( (\log y )^2 )$, we have $S(x,y,q) \ll (\log x) \exp (- \varepsilon^2 u/8)$, which gives a very large saving whenever $\log \log x = o(u)$. If $y$ is smaller than roughly $e^{\sqrt{\log x}}$, the first term in the parentheses begins to dominate. In that range, it may be helpful to note that~\eqref{equ: saddle point est} implies that $\alpha(x,y) \geq \frac{5}{6} (1+o(1))$ uniformly on the entire range of $y$ considered, from which it follows that $(\log x)^{3/2} (\log y)^{1/2} y^{1/2 - \alpha (x,y)} = o(1)$ uniformly for $y$ in our range. We also note that this theorem has no content (in the sense that it is worse than the ``trivial bound'' of $\sqrt{\Psi (x,y)}$ obtained by applying Cauchy--Schwarz and orthogonality) whenever $\frac{\log q}{\log x} \leq 1 + \bigl(\frac{6 \log \log x + 2 \log \log y + C}{u}\bigr)^{1/2}$ for any constant $C > 0$. The need to have $q$ a little larger than $x$ to obtain a non-trivial saving in the proof is related to the anatomy of $y$-smooth numbers, which we explain more in the outline of our proof in Section~\ref{s: outline}. With this discussion in mind, we note that the theorem is actually proved on the range $(\log x)^4 \leq y \leq x^{\frac{1}{4 \log \log x}}$, though it is often worse than the ``trivial bound'' at the extremes of this range.

To the best of the authors' knowledge, Theorem~\ref{t: main} is the first result in the literature to provide non-trivial insight into the first absolute moment of character sums over smooth numbers. Another direction is to bound the count of characters whose sums over smooth numbers are extremely large (i.e.\,almost the size of the trivial bound). This is done in works of~\citet{DrappeauGranvilleShao} and~\citet{Shparlinski}.

\begin{xremark}
\normalfont It seems possible to combine ideas from the character sum work of~\citet{HarperTypical} with~\cite[Theorem~1.3]{HardyXu} to prove that for $q$ prime, $1 \leq x \leq q$, and $\exp ((\log x)^{1/5}) \leq y \leq x$, say, we have $\E_{\chi (q)} \bigl| \sum_{n \leq x, P(n) \leq y} \chi(n) \bigr| \ll \sqrt{\Psi(x,y)} (\log \log (10 + \min \{ x, q/x\}))^{-1/4}$. It would be interesting if one could also obtain a saving over the trivial bound when $y$ is small, say $y \leq (\log x)^{6}$, as is done in the random multiplicative setting in~\cite[Theorem~1.4]{HardyXu}.
\end{xremark}

The bound obtained in Theorem~\ref{t: main} is weaker than the bound obtained in the random multiplicative case~\cite[Theorem~1.2]{HardyXu}, which we believe to be close to optimal, though there we benefit immensely from the genuine independence of the Steinhaus random multiplicative function on primes. In contrast to previous works~\cite{CaichShortInt, HarperTypical}, the proof of this result \emph{does not} directly compare the problem to the random multiplicative case, though the random multiplicative case does serve to inspire the proof strategy. We provide an outline of the proof in Section~\ref{s: outline}, first stating some generalisations and defining the notation used throughout the paper.

\subsection{Generalisations} The proof generalises to other cases, including when we introduce completely multiplicative twists to the sums, and also when we average over continuous characters. In the first case, we prove the following:

\begin{theorem}[Completely multiplicative twists]\label{t: mult twists}
Let $S$ be the saving factor defined in Theorem~\ref{t: main}. Suppose that $h$ is a completely multiplicative function satisfying $|h(p)| \leq 1$ for every prime $p$. Then for $2 \leq x \leq q$ and $(\log x)^6 \leq y \leq x^{\frac{1}{32 \log \log x}}$, we have
\[
\frac{1}{\varphi(q)} \sum_{\chi \bmod q} \Bigl| \sum_{\substack{n \leq x \\ P(n) \leq y}} h(n) \chi(n) \Bigr| \ll \sqrt{\Psi(x,y) S(x,y,q)}.
\]
\end{theorem}

An exemplary function that this result can be applied to is the Liouville function, $ h = \lambda$, defined as the unique completely multiplicative function such that $\lambda (p) = -1$ for all primes $p$. We note that the upper bound is written in terms of $\sqrt{\Psi(x,y)}$, though in this case, the Cauchy--Schwarz bound is $\sum_{n \leq x, P(n) \leq y} |h(n)|^2$, which can be smaller than $\sqrt{\Psi (x,y)}$. Therefore, the result should only be applied when $\sum_{n \leq x , \, P(n) \leq y} |h(n)|^2 \approx \Psi (x,y)$. The result can likely be generalised to other multiplicative twists, including, for example, the Möbius function, though proving such a result introduces some technical difficulties\footnote{Specifically, since $\mu$ is not completely multiplicative, the coefficients of $\chi(p)^j \mu(p)^j$ in the expansion corresponding to~\eqref{equ: full expansion} are no longer positive, so we cannot easily compare the truncated expansion to $\sum_{n \leq y^k, P(n) \leq y} \frac{\mu (n) \chi (n)}{n^{\alpha/2 + it}}$, as in Lemma~\ref{l: good generating function}.} that we do not resolve for the sake of exposition. We now state generalisations of the result to continuous characters and Dirichlet characters with additional averaging over the conductor:

\begin{theorem}[Other families]\label{t: other families}
Let $S$ be the saving factor defined in Theorem~\ref{t: main}. Suppose that $(\log x)^6 \leq y \leq x^{\frac{1}{32 \log \log x}}$. For $x \leq T$, we have
\[
\frac{1}{T} \int_{T}^{2T} \Bigl| \sum_{\substack{n \leq x \\ P(n) \leq y}} n^{it} \Bigr| \, dt \ll \sqrt{\Psi (x,y) S(x,y, T)}.
\]
Furthermore, letting $\sum\nolimits^{*}_{\chi \bmod q}$ denote the sum over primitive Dirichlet characters $\chi \bmod q$, for $x \leq Q^2$, we have
\[
\frac{1}{Q^2} \sum_{q \leq Q} \, \, \sideset{}{^*}\sum_{\chi \bmod q} \Bigl| \sum_{\substack{n \leq x \\ P(n) \leq y}} \chi(n) \Bigr| \ll \sqrt{\Psi (x,y) S(x,y, Q^2)}.
\]
The same results also hold with multiplicative twists of the form in Theorem~\ref{t: mult twists}.
\end{theorem}

Besides its own interest, the first estimate in this theorem is related to the famous open problem of finding $x^{\varepsilon}$-smooth numbers in the short interval $[x, x+\sqrt{x}]$ when $x$ is sufficiently large, though here the smoothness parameter is smaller than the range there, and one needs to increase the range of $x$ relative to $T$ to make progress on this problem. This is discussed more in the introduction to~\cite{HardyXu}. See the work of~\citet{MR16} for the best-known result on this problem.

The proofs of Theorems~\ref{t: mult twists} and~\ref{t: other families} are almost identical to the proof of Theorem~\ref{t: main}. We outline the necessary adjustments in Section~\ref{s: other families}.

\subsection{Notation}

We employ probabilistic notation throughout the paper, always taking the uniform probability measure on the spaces considered. For example, in the case of Dirichlet characters modulo $q$ (which we consider for most of the paper) we define
\[
\p_{\chi (q)} (E) \coloneqq \frac{1}{\varphi(q)} \# \bigl\{ \chi \bmod q : \chi \in E \bigr\},
\]
for any $E \subseteq \{ \chi \bmod q \}$. As is standard in probabilistic notation, we may write conditions inside the probability, which represent the set of characters satisfying such conditions. For any function $W \colon \{ \chi \bmod q \} \rightarrow \C$, we have the expectation
\[
\E_{\chi (q)} W(\chi) \coloneqq \frac{1}{\varphi(q)} \sum_{\chi \bmod q} W(\chi).
\]
We also often use $\alpha$ as shorthand for the saddle point, $\alpha (x,y)$, defined in~\eqref{equ: saddle defn}, and frequently use the notation $u = \frac{\log x}{\log y}$.

\subsection*{Acknowledgements} 
The authors are grateful to Dimitris Koukoulopoulos for a correspondence related to the anatomy of smooth numbers. SH is grateful to Adam Harper for useful discussions on this problem. This work was completed while the authors were visiting the Centre de Recherches Mathématiques in Montréal; we thank them for excellent working conditions. SH is supported by the Swinnerton-Dyer scholarship at the Warwick Mathematics Institute Centre for Doctoral Training. MWX was supported by a Simons Junior Fellowship from the Simons Foundation.

\section{Deducing Theorem~\ref{t: main} from two key lemmas}

\subsection{Outline of the proof}\label{s: outline}

The goal of the proof is to imitate the argument used in the setting of random multiplicative functions~\cite[Theorem~1.2]{HardyXu}, though interesting difficulties arise from doing this. The proof there is conceptually straightforward, and involves applying Perron's formula followed by the triangle inequality, which leads one to evaluate expectations of random Euler products. In the setting of sums of Dirichlet characters, the corresponding Euler products are
\[
L (s, \chi, y) \coloneqq \sum_{P(n) \leq y} \frac{\chi(n)}{n^{s}} = \prod_{p \leq y} \biggl( 1 - \frac{\chi (p)}{p^{s}} \biggr)^{-1} .
\]
Throughout the paper, we also employ the standard notation $\zeta (s, y) \coloneqq \prod_{p \leq y} \bigl( 1 - \frac{1}{p^{s}} \bigr)^{-1}$. Applying Perron's formula\footnote{In the random multiplicative case~\cite[Section~4]{HardyXu}, the integral is shifted to the line $\mathrm{Re} (s) = \alpha (x^2, y) / 2$. In that work, this gives a larger saving, but here this would decrease the size of the saving in Lemma~\ref{l: good generating function}, and thus decrease the overall size of our saving.} and shifting to the line $\mathrm{Re} (s) = \alpha(x,y) /2$, then applying the triangle inequality, we obtain
\[
\E_{\chi (q)} \Bigl| \sum_{\substack{n \leq x \\ P(n) \leq y}} \chi(n) \Bigr| \leq x^{\alpha/2} \int_{-x}^{x} \frac{\E_{\chi (q) } \bigl| L ( \alpha + it, \chi, y ) \bigr|}{|\alpha/2+it|} \, dt + O(\log x),
\]
where we write $\alpha = \alpha (x,y)$. However, in contrast to the random multiplicative case, it is not obvious that we can obtain any meaningful upper bound for the quantity
\[
\E_{\chi(q)} \bigl| L \bigl( \alpha + it, \chi, y \bigr) \bigr|.
\]
In particular, the values of $(\chi(p))_{p \leq y}$ are not independent, so we cannot simply use the Euler product formula to write this as a product of expectations. To orient the reader, in the random multiplicative case, the analogous quantity can be shown~\cite[Lemma~3.1]{HardyXu} to be $\ll \zeta(\alpha , y)^{1/4}$, which gives a large saving (of size roughly $\zeta(\alpha,y)^{1/4} \approx e^{u/4}$ on our range of $y$) over the Cauchy--Schwarz bound. We observe that it is not necessary to understand $\E_\chi | L ( \alpha + it, \chi, y )|$, and we can choose a different Dirichlet series in the Perron integral instead: one for which we can bound the first absolute moment. For $\mathrm{Re}(s) > 0$, by Taylor expansion of the logarithm, we have
$- \sum_{p \leq y} \log \bigl( 1 - \frac{\chi(p)}{p^{s}} \bigr) = \sum_{j \geq 1} \frac{1}{j} \sum_{p \leq y} \frac{\chi(p)^j}{p^{j s}}$, from which it follows that
\begin{equation}\label{equ: full expansion}
\sum_{\substack{P(n) \leq y}} \frac{\chi(n)}{n^s} = L \bigl( s, \chi, y \bigr) = \sum_{k=0}^{\infty} \frac{1}{k!} \biggl( \sum_{j \geq 1} \frac{1}{j} \sum_{p \leq y} \frac{\chi(p)^j}{p^{j s}} \biggr)^{k} .
\end{equation}
An object that we \emph{can} obtain meaningful estimates for is the first absolute moment of the truncated Taylor expansion,
\[
\sum_{k=0}^{K} \frac{1}{k!} \Bigl( \sum_{j=1}^{\lfloor \frac{\log y}{\log 2} \rfloor} \frac{1}{j} \sum_{p \leq y^{1/j}} \frac{\chi(p)^j}{p^{j(\alpha/2 + it)}} \Bigr)^k ,
\]
so long as $y^K \leq q$, since then we have access to ``perfect'' orthogonality for characters $\chi \bmod q$. Specifically, we can do this by showing that the term in the parentheses is relatively small for almost all characters $\chi \bmod q$, which is the contents of Lemma~\ref{l: large value bound}. In order for this argument to work, we need this truncated expansion to serve as a sufficiently good generating function for the partial sums $\sum_{n \leq x, \, P(n) \leq y} \chi (n)$. A typical $y$-smooth number below $x$ has approximately $u = \frac{\log x}{\log y}$ prime factors in the regime $(\log x)^2 \leq y \leq x^{o(1/\log \log x)}$ (this follows from Erdős--Kac-type results for smooth numbers, see~\citet{Alladi}), so we need to take $K > u(1+\varepsilon)$ for the generating function to capture typical $y$-smooth numbers. This is ultimately the reason why we need $q \geq x^{1 + \varepsilon}$ in the proof, since the condition $y^K \leq q$ needed for orthogonality then enforces $x^{1 + \varepsilon} \leq q$. To show that the previous display serves as a sufficiently good generating function, we compare it to
\[
\sum_{\substack{n \leq y^K \\ P(n) \leq y}} \frac{\chi(n)}{n^{\alpha/2+it}},
\]
which we can certainly use in place of $L ( \alpha + it, \chi, y )$ in the Perron integral whenever $y^K \geq x$. This is the contents of Lemma~\ref{l: good generating function}.

\subsection{Lemmas}
In this section, we will state the lemmas mentioned in the previous section. We first introduce some definitions: for the remainder of the paper, we define
\begin{equation}\label{equ: defn of P}
P (s, \chi, y) \coloneqq \sum_{j = 1}^{\lfloor \frac{\log y}{\log 2} \rfloor} \frac{1}{j} \sum_{p \leq y^{1/j}} \frac{\chi(p)^j}{p^{js}}.
\end{equation}
As mentioned, this will essentially serve as a proxy for $\log L ( \alpha + it, \chi, y )$ whilst also being of length $\leq y$ (in the sense that $\chi$ is only ever evaluated at integers smaller than $y$). For $K \in \N$, we have
\begin{equation}\label{equ: truncated expansion}
\sum_{k=0}^{K} \frac{P(\alpha/2 + it, \chi, y)^k}{k!} = \sum_{k=0}^{K} \frac{1}{k!} \biggl( \sum_{j = 1}^{\lfloor \frac{\log y}{\log 2} \rfloor} \frac{1}{j} \sum_{p \leq y^{1/j}} \frac{\chi(p)^j}{p^{j(\alpha / 2 + it)}} \biggr)^k = \sum_{\substack{n \leq y^K \\ P(n) \leq y}} \frac{a(n) \chi(n)}{n^{\alpha/2+it}},
\end{equation}
where $a(n)$ satisfy the following properties:
\begin{equation}\label{equ: a(n) properties}
a(n) = 1 \text{ if } p^j \mid n \Rightarrow p \leq y^{1/j} \text{ \emph{and} } \omega(n) \leq K. \text{ Otherwise, } 0 \leq a(n) \leq 1.
\end{equation}
These properties follow by comparing coefficients in equation~\eqref{equ: truncated expansion} to those in the full expansion~\eqref{equ: full expansion}, and noting that the expansions agree for the coefficient of $n$ if $n$ satisfies both $p^j \mid n \Rightarrow p \leq y^{1/j}$ and $\omega(n) \leq K$. We now state the two main lemmas that allow us to prove Theorem~\ref{t: main}, and whose proofs we defer to Section~\ref{s:proof of lemmas}. 
\pagebreak
\begin{lemma}[Approximation by prime sum expansion]\label{l: good generating function}
Suppose that $(\log x)^4 \leq y \leq x$, let $u = \frac{\log x}{\log y}$, and let $\alpha = \alpha (x,y) $ be the saddle point. If $K \in \N$ is such that $x \leq y^K \leq \min \{ q, x^{3/2} \}$, then uniformly for $t \in \R$, we have
\[
\E_{\chi (q)} \biggl| \sum_{\substack{n \leq y^K \\ P(n) \leq y}} \frac{\chi(n)}{n^{\alpha / 2 + it}} - \sum_{k = 0}^{K} \frac{P (\alpha / 2 + it, \chi, y)^{k}}{k!} \biggr| \ll E (x, y, K), 
\]
where
\[
E (x, y, K) \coloneqq \zeta (\alpha, y)^{1/2} \biggl(y^{(1/2 - \alpha)/2} + \exp \biggl\{ - \frac{u}{8} \biggl( \frac{K}{u} -1 \biggr)^2 \biggr\} \biggr).
\]
\end{lemma}
This lemma is responsible for the form of the saving in Theorem~\ref{t: main}. Importantly, when $K$ is a little larger than $u$, this bound will be significantly smaller than $\zeta (\alpha, y)^{1/2}$, which is the Cauchy--Schwarz bound for the size of the generating function.

\begin{xremark}
\normalfont The first error in the parentheses comes from the fact that we do not capture prime factors that occur with high multiplicity (since $p \leq y^{1/j}$ in the definition of $P(s,\chi,y)$ in~\eqref{equ: defn of P}). This limits the size of the saving when $y$ is small, since in this case, prime factors typically appear with higher multiplicity. The second error in the parentheses captures the fact that the approximation only sees integers with fewer than $K$ prime factors, and is consistent with the fact that the count of prime factors of a $y$-smooth number follows a Gaussian distribution~\cite{Alladi} (with more effort, the constant could be taken to be $1/4$ instead of $1/8$, but we do not pursue this improvement here). In practice, we obtain this saving from Rankin's trick, which is morally equivalent to applying a Chernoff bound. An alternative method would be to remove integers with fewer than $K$ prime factors before applying Perron's formula.
\end{xremark}

\begin{lemma}[Tail bound for sum over primes]\label{l: large value bound}
Suppose that $(\log x)^4 \leq y \leq x^{\frac{1}{4 \log \log x}}$, let $u = \frac{\log x}{\log y}$, and let $\alpha = \alpha (x,y)$ be the saddle point. There exists an absolute constant $c \geq 1$ so that for $V \geq c \sqrt{u}$ where $u=\frac{\log x}{\log y}$, and for any $t \in \R$, we have
\[
\p_{\chi(q)} \bigl( |P(\alpha / 2 + it, \chi , y) | > V \bigr) \ll \sqrt{k} \biggl( \frac{k u}{2 V^2} \biggr)^{k} ,
\]
uniformly for $k \in \N$ satisfying $k \leq \lfloor \frac{\log q}{\log y} \rfloor$.
\end{lemma}

\begin{xremark}
\normalfont This lemma is more general than necessary; in applications, we will only apply it with $V \asymp u$. One could obtain a better estimate with $1/e$ in place of $1/2$, but this would involve tighter restrictions on the upper range of $y$.
\end{xremark}

\subsection{Proof of Theorem~\ref{t: main}, assuming Lemmas~\ref{l: good generating function} and~\ref{l: large value bound}}\label{s: Proof of main theorem}
In this section, we will assume Lemmas~\ref{l: good generating function} and~\ref{l: large value bound}; as mentioned, their proofs are deferred to Section~\ref{s:proof of lemmas}. Theorem~\ref{t: main} will be an immediate consequence of the following proposition, which is a straightforward application of the lemmas.
\begin{proposition}\label{p: L1 norm}
Suppose that $x \leq q$ and $(\log x)^4 \leq y \leq x^{\frac{1}{4 \log \log x}}$. Let $\alpha = \alpha (x,y)$ be the saddle point, let $u = \frac{\log x}{\log y}$, and define $\mathcal{K} \coloneqq \lfloor \frac{\log \min \{ x^{3/2}, q \}}{\log y} \rfloor$. Uniformly for any $t \in \R$, we have
\[
\E_{\chi (q)} \biggl| \sum_{\substack{n \leq y^{\mathcal{K}} \\ P(n) \leq y}} \frac{\chi(n)}{n^{\alpha / 2 + it}} \biggr| \ll \zeta (\alpha, y)^{1/2} \biggl(y^{(1/2 - \alpha)/2} + \exp \biggl\{ - \frac{u}{8} \biggl( \frac{\log \min\{ q, x^{3/2} \}}{\log x} - 1 \biggr)^2 \biggr\} \biggr).
\]
\end{proposition}

\begin{xremark}
\normalfont Since we wish to apply Lemma~\ref{l: good generating function}, it is convenient for the length of the sum to be an integer power of $y$. On first reading, the sum should be thought of as going up to $\min \{ x^{3/2}, q \}$.
\end{xremark}

\begin{proof}[Proof of Proposition~\ref{p: L1 norm}, assuming Lemmas~\ref{l: good generating function} and~\ref{l: large value bound}. ]
By linearity of expectation, the term that we wish to bound is equal to
\[
\E_{\chi (q)} \mathbf{1}_{|P (\alpha / 2 + it, \chi, y)| \leq V} \Bigl| \sum_{\substack{n \leq y^{\mathcal{K}} \\ P(n) \leq y}} \frac{\chi(n)}{n^{\alpha /2+it}} \Bigr| + \E_{\chi (q)} \mathbf{1}_{|P (\alpha / 2 + it, \chi, y)| > V} \Bigl| \sum_{\substack{n \leq y^{\mathcal{K}} \\ P(n) \leq y}} \frac{\chi(n)}{n^{\alpha /2+it}} \Bigr| .
\]
Applying the Cauchy--Schwarz inequality, the second term in the previous display is bounded above by
\[
\p_{\chi (q)} \bigl( |P (\alpha / 2 + it, \chi, y)| > V \bigr)^{1/2} \biggl( \E_{\chi (q)} \Bigl| \sum_{\substack{n \leq \min \{x^{3/2}, q\} \\ P(n) \leq y}} \frac{\chi(n)}{n^{\alpha / 2 + it}} \Bigr|^2 \biggr)^{1/2}
\]
By applying orthogonality and dropping the upper bound condition on the size of $n$, we see that the expectation inside the parentheses is $\ll \zeta(\alpha, y)$. To bound the probability, we apply Lemma~\ref{l: large value bound} with $V = u/3$ and $k = 2 u / 9e$, to find the previous display is
\[
\ll k^{1/4} \biggl( \frac{k u}{2 V^2} \biggr)^{k/2} \zeta(\alpha, y)^{1/2} \ll u^{1/4} e^{-u/9e} \zeta(\alpha, y)^{1/2} \ll e^{-u/25} \zeta(\alpha, y)^{1/2} .
\]
We note that there is considerable flexibility in these choices, and since the right-hand side of Proposition~\ref{p: L1 norm} is always $\gg \zeta (\alpha, y)^{1/2} e^{-u/32}$, this bound is acceptable. It remains to show that
\begin{multline*}
\E_{\chi (q)} \mathbf{1}_{|P (\alpha/2 + it, \chi, y)| \leq V} \Bigl| \sum_{\substack{n \leq y^{\mathcal{K}} \\ P(n) \leq y}} \frac{\chi(n)}{n^{\alpha / 2 + it}} \Bigr| \\ 
\ll \zeta (\alpha, y)^{1/2} \biggl(y^{(1/2 - \alpha)/2} + \exp \biggl\{ - \frac{u}{8} \biggl( \frac{\log \min\{ q, x^{3/2} \}}{\log x} -1 \biggr)^2 \biggr\} \biggr),
\end{multline*}
when $V = u/3$. By the triangle inequality, we can upper-bound the left-hand side by
\[
\E_{\chi(q)} \mathbf{1}_{|P (\alpha/2 + it, \chi, y)| \leq V} \Bigl| \sum_{k=0}^{\mathcal{K}} \frac{P (\alpha/2 + it, \chi, y)^k}{k!} \Bigr| + \E_{\chi (q)} \biggl| \sum_{\substack{n \leq y^{\mathcal{K}} \\ P(n) \leq y}} \frac{\chi(n)}{n^{\alpha / 2 + it}} - \sum_{k = 0}^{\mathcal{K}} \frac{P (\alpha / 2 + it, \chi, y)^{k}}{k!} \biggr|,
\]
where $\mathcal{K} = \lfloor \frac{\log \min \{ x^{3/2}, q \}}{\log y} \rfloor$. To handle the first term, we apply the triangle inequality and make use of the condition that $|P (\alpha/2 + it, \chi, y)| \leq V = u/3$, delivering a bound of $\ll e^{u/3}$ by completing the sum. This is $\ll \zeta (\alpha, y)^{1/2} e^{-u/6}$ (which follows immediately from the fact that $u \leq \zeta(\alpha, y) + c$ for some absolute constant $c>0$, see~\eqref{equ: u upper bound log zeta}), and seeing as the bound we are aiming for is always $\gg \zeta (\alpha, y)^{1/2} e^{-u/32}$, this is admissible. It remains to handle the second term in the previous display. We bound this using Lemma~\ref{l: good generating function}, giving
\begin{multline*}
\E_{\chi (q)} \biggl| \sum_{\substack{n \leq y^{\mathcal{K}} \\ P(n) \leq y}} \frac{\chi(n)}{n^{\alpha / 2 + it}} - \sum_{k = 0}^{\mathcal{K}} \frac{P (\alpha / 2 + it, \chi, y)^{k}}{k!} \biggr| \\ 
\ll \zeta (\alpha, y)^{1/2} \biggl(y^{(1/2 - \alpha)/2} + \exp \biggl\{ - \frac{u}{8} \biggl( \frac{\mathcal{K}}{u} -1 \biggr)^2 \biggr\} \biggr) .
\end{multline*}
Then applying the fact that $\mathcal{K}/u = \frac{\log \min\{ q, x^{3/2} \}}{\log x} + O(1/u)$, we complete the proof of Proposition~\ref{p: L1 norm}.
\end{proof}

Proposition~\ref{p: L1 norm} gives an upper bound for the first absolute moment of a generating function of the original sum, which crucially beats the Cauchy--Schwarz bound of size $\zeta(\alpha,y)^{1/2}$. To deduce Theorem~\ref{t: main}, we can now just apply Perron's formula followed by the triangle inequality.

\begin{proof}[Proof of Theorem~\ref{t: main}, assuming Lemmas~\ref{l: good generating function} and~\ref{l: large value bound}]
Suppose that $(\log x)^4 \leq y \leq x^{\frac{1}{4 \log \log x}}$, and recall that $\mathcal{K} = \lfloor \frac{\log \min \{ x^{3/2}, q \}}{\log y} \rfloor$. We can assume that $q \geq xy$, since Theorem~\ref{t: main} is worse than the trivial Cauchy--Schwarz bound when this is not the case (see the discussion following the statement of Theorem~\ref{t: main}). This means that $x \geq y^{\mathcal{K}}$, and therefore, by Perron's formula~\cite[Corollary~5.3]{MV2007}, we have
\begin{multline*}
\sum_{\substack{n \leq x \\ P(n) \leq y}} \chi(n) = \frac{1}{2 \pi i} \int_{\alpha/2 - ix}^{\alpha/2 + ix} \sum_{\substack{n \leq y^{\mathcal{K}} \\ P(n) \leq y}} \frac{\chi(n)}{n^s} x^s \frac{ds}{s} \\  
+ O \biggl( 1 + \sum_{\substack{x/2 < n \leq 2x \\ n \neq x}} \min \biggl\{ 1, \frac{1}{|x-n|} \biggr\} + \frac{x^{\alpha/2}}{x} \sum_{n \leq \min \{x^{3/2}, q\}} \frac{1}{n^{\alpha/2}} \biggr).
\end{multline*}
The error term here is $\ll x^{1/2 - \alpha/4} \log x$, since the first sum in the error term is bounded by $\ll 1 +  \sum_{1\le i \le x}i^{-1}\ll \log x$ and the second sum is $\ll x^{1/2-\alpha/4}$. Therefore, applying the triangle inequality, we have
\[
\E_{\chi (q)} \Bigl| \sum_{\substack{n \leq x \\ P(n) \leq y}} \chi(n) \Bigr| \ll x^{\alpha / 2} \int_{-x}^{x} \E_{\chi (q)} \Bigl| \sum_{\substack{n \leq y^{\mathcal{K}} \\ P(n) \leq y}} \frac{\chi(n)}{n^{\alpha / 2 + it}} \Bigr| \frac{dt}{|\alpha / 2 + it|} + x^{1/2 - \alpha/4} \log x .
\]
Applying Proposition~\ref{p: L1 norm}, the right-hand side is
\[
\ll x^{\alpha/2}  \zeta (\alpha, y)^{1/2} (\log x) \biggl(y^{(1/2 - \alpha)/2} + \exp \biggl\{ - \frac{u}{8} \biggl( \frac{\log \min\{ q, x^{3/2} \}}{\log x} - 1 \biggr)^2 \biggr\} \biggr) + x^{1/2 - \alpha/4} \log x .
\]
We upper bound the first term using the result of~\citet[Theorem~1]{HildTen86} that $x^{\alpha} \zeta (\alpha, y) \leq \Psi(x,y) \sqrt{\log x \log y}$. Furthermore, using~\eqref{equ: saddle point est} to bound the second term and using the fact (see~\cite[Corollary to Theorem~3.1]{CEP}) that $\Psi(x,y) = x e^{-u \log u (1+o(1))}$ uniformly on our range of $y$ to lower bound the first term, one can find that the first term dominates the size of this expression. Therefore, we deduce that
\[
\E_{\chi (q)} \Bigl| \sum_{\substack{n \leq x \\ P(n) \leq y}} \chi(n) \Bigr| \ll \sqrt{\Psi(x,y) S(x,y,q)},
\]
where
\[
S(x,y,q) = (\log x)^{3/2} (\log y)^{1/2} \biggl( y^{1/2 - \alpha(x,y)} +  \exp \biggl\{ - \frac{u}{4} \biggl( \frac{\log \min\{ q, x^{3/2} \}}{\log x} - 1 \biggr)^2  \biggr\} \biggr),
\]
as required.
\end{proof}

\section{Proofs of Lemmas}\label{s:proof of lemmas}

\subsection{Useful estimates} 
In this section, we prove Lemmas~\ref{l: good generating function} and~\ref{l: large value bound}. To prove these, it will be helpful to have the following estimate: with $\alpha = \alpha (x,y)$ the saddle point, for $(\log x)^2 \leq y \leq x^{1/3}$, we have
\begin{equation}\label{equ: prime sum}
\sum_{p \leq y} \frac{1}{p^{\alpha }} \leq u + \log \log x + O \biggl( \frac{u \log \log u}{\log u} \biggr) .
\end{equation}
This follows by combining a prime number sum estimate (for example~\cite[Lemma~7.4]{MV2007}) with estimates of~\citet[Lemmas~2 and~3]{HildTen86} for the saddle point that are slightly stronger than~\eqref{equ: saddle point est}. We will also apply the following simple but useful inequality: uniformly for $(\log x)^4 \leq y \leq x$, there exists an absolute constant $c \geq 0$ such that
\begin{equation}\label{equ: u upper bound log zeta}
u \leq \log \zeta (\alpha, y) + c.
\end{equation}
This suffices for our purposes, though we note that it actually holds with $c = 0$. The proof follows from the definition of the saddle point~\eqref{equ: saddle defn},
\[
u = \frac{\log x}{\log y} = \sum_{p \leq y} \frac{\log p}{\log y} \frac{1}{p^{\alpha} - 1} = \sum_{k \geq 1} \sum_{p \leq y} \frac{\log p}{p^{k \alpha} \log y} \leq \sum_{p \leq y} \frac{1}{p^\alpha} + O(1) = \log \zeta(\alpha, y) + O(1).
\]
Here we have vitally used~\eqref{equ: saddle point est} (and the fact that the saddle point is increasing in $y$) to note that uniformly for $(\log x)^4 \leq y \leq x$, we have $\alpha(x,y) \geq 2/3$, say, so that $\sum_{p \leq y} \frac{1}{p^{2\alpha}} \ll 1$.

\subsection{Proof of Lemma~\ref{l: good generating function}}
We first give the proof of Lemma~\ref{l: good generating function}, which is arguably the most important part of the proof, since it is the limiting factor in improving our result.
\begin{proof}[Proof of Lemma~\ref{l: good generating function}]
Suppose that $(\log x)^4 \leq y \leq x$. Let $K \in \N$ be such that $x \leq y^K \leq \min \{ q, x^{3/2} \}$, and take any $t \in \R$. By Cauchy--Schwarz, it suffices to show that
\[
\biggl( \E_{\chi (q)} \biggl| \sum_{\substack{n \leq y^K \\ P(n) \leq y}} \frac{\chi(n)}{n^{\alpha / 2 + it}} - \sum_{k = 0}^{K} \frac{P (\alpha / 2 + it, \chi, y)^{k}}{k!} \biggr|^2 \biggr)^{1/2} \ll E (x, y, K),
\]
where
\[
E (x, y, K) = \zeta (\alpha(x,y), y)^{1/2} \biggl(y^{(1/2 - \alpha(x,y))/2} + \exp \biggl\{ - \frac{u}{8} \Bigl( \frac{K}{u} -1 \Bigr)^2 \biggr\} \biggr).
\]
Rewriting the sum over $k$ using~\eqref{equ: truncated expansion} and applying orthogonality of characters (crucially using the fact that $y^K \leq q$), we find that
\[
\biggl( \E_{\chi (q)} \biggl| \sum_{\substack{n \leq y^K \\ P(n) \leq y}} \frac{\chi(n)}{n^{\alpha / 2 + it}} - \sum_{k = 0}^{K} \frac{P (\alpha / 2 + it, \chi, y)^{k}}{k!} \biggr|^2 \biggr)^{1/2} \leq \biggl( \sum_{\substack{n \leq y^K \\ P(n) \leq y}} \frac{|1 - a(n)|^2}{n^{\alpha}} \biggr)^{1/2} .
\]
Now using the properties of $a(n)$ from~\eqref{equ: a(n) properties}, any integers $n$ in the sum satisfying $a(n) \neq 1$ must be either divisible by an integer of the form $p^j \geq y$, \emph{or} satisfy $\omega (n) \geq K$. When either of these conditions occur for an integer $n$, we can always apply the bound $|1 - a(n)|^2 \leq 1$, so the previous display is
\[
\leq \biggl( \sum_{j=2}^{K \log y} \sum_{p \geq y^{1/j}} \sum_{\substack{n \leq y^K \\ P(n) \leq y \\ p^j \mid n}} \frac{1}{n^{\alpha}} + \sum_{\substack{n \leq y^K \\ P(n) \leq y \\ \omega(n) \geq K}} \frac{1}{n^{\alpha}} \biggr)^{1/2}.
\]
By subadditivity of $x \rightarrow x^{1/2}$, to complete the proof, it suffices to show that
\[
\sum_{j=2}^{K \log y} \sum_{p \geq y^{1/j}} \sum_{\substack{n \leq y^K \\ P(n) \leq y \\ p^j \mid n}} \frac{1}{n^{\alpha}} \ll \zeta (\alpha, y) y^{(1/2 - \alpha)} ,
\]
and that
\begin{equation}\label{equ: divisors large deviation}
\sum_{\substack{n \leq y^K \\ P(n) \leq y \\ \omega(n) \geq K}} \frac{1}{n^{\alpha}} \ll \zeta(\alpha, y) \exp \biggl\{ - \frac{u}{4} \biggl( \frac{K}{u} -1 \biggr)^2 \biggr\} .
\end{equation}
The first statement is straightforward to prove. We have
\[
\sum_{j=2}^{K \log y} \sum_{p \geq y^{1/j}} \sum_{\substack{n \leq y^K \\ P(n) \leq y \\ p^j \mid n}} \frac{1}{n^{\alpha}} \ll \sum_{j=2}^{K \log y} \sum_{p \geq y^{1/j}} \frac{\zeta(\alpha, y)}{p^{j \alpha}} \ll \frac{\zeta(\alpha, y)}{\log y} \sum_{j = 2}^{K \log y} j y^{1/j - \alpha} \ll \zeta (\alpha, y) y^{1/2 - \alpha},
\]
as required, so we just need to prove~\eqref{equ: divisors large deviation}. First, applying Rankin's trick, we have
\[
\sum_{\substack{n \leq y^K \\ P(n) \leq y \\ \omega(n) \geq K}} \frac{1}{n^{\alpha}} \leq  e^{-cK} \sum_{\substack{n \leq y^K \\ P(n) \leq y}} \frac{e^{c \omega (n)}}{n^{\alpha}} = e^{-cK} \prod_{p \leq y} \biggl( 1 + \frac{e^c}{p^{\alpha}} + O \bigl( 1/p^{2 \alpha} \bigr) \biggr) \ll \exp \biggl\{ \sum_{p \leq y} \frac{e^c}{p^{\alpha}} - cK \biggr\},
\]
uniformly for any bounded $c>0$. Here we have used~\eqref{equ: saddle point est} to note that $\alpha \geq 2/3 $ uniformly for $y \geq (\log x)^4$, say, and hence $\sum_{p \leq y} \frac{1}{p^{2\alpha}} \ll 1$. Taking $c = \log (K / u)$ roughly minimises this error, though to be precise, the error is minimised by instead taking it to be $\log ( K / \sum_{p \leq y} \frac{1}{p^{\alpha}})$. In light of~\eqref{equ: prime sum}, this will not matter on the range of $y$ we end up considering, and an advantage of taking this value of $c$ is that it is immediately from the condition $x \leq y^K$ that $c>0$. We obtain
\[
\sum_{\substack{n \leq y^K \\ P(n) \leq y \\ \omega(n) \geq K}} \frac{1}{n^{\alpha}} \leq \exp \biggl\{ \frac{K u}{\sum_{p \leq y} \frac{1}{p^{\alpha}}} \Bigl( 1 - \log \bigl( K / u \bigr) \Bigr) \biggr\},
\]
and we proceed to bound this term more explicitly. Seeing as $K$ will be of size comparable to $u$, we define $\delta = \delta (K)$ to be such that $\delta = \frac{K}{u}-1$ (equivalently $K = (1 + \delta) u$). The condition $x \leq y^K \leq \min \{ q, x^{3/2} \}$ implies that $\delta \in [0, 1/2]$. Using the fact that $1 - \log (1+x) \leq 1 - x + x^2/2$ for $x \geq 0$, we then have
\begin{align*}
K \Bigl( 1 - \log \bigl( K / u \bigr) \Bigr) &= u (1 + \delta) \bigl( 1 - \log (1 + \delta) \bigr) \\
& \leq u (1 + \delta) (1 - \delta + \delta^2/2) \\
& = u (1 - \delta^2/2 +\delta^3/2) \\
& \leq u (1 - \delta^2/4),
\end{align*}
where, to obtain the last line, we have made use of the fact that $-x^2/2 + x^3/2 \leq -x^2/4$ for $x \in [0,1/2]$. Inserting this bound into the previous display, we have shown that
\[
\sum_{\substack{n \leq y^K \\ P(n) \leq y \\ \omega(n) \geq K}} \frac{1}{n^{\alpha}} \ll \exp \biggl\{ \frac{u^2}{\sum_{p \leq y} \frac{1}{p^{\alpha}}} \biggl( 1 - \frac{\delta^2}{4} \biggr) \biggr\}.
\]
Finally, making use of the facts that $\log \zeta (\alpha, y) = \sum_{p \leq y} \frac{1}{p^\alpha} + O(1)$ and $u \leq \log \zeta(\alpha,y) + c$ (as shown in~\eqref{equ: u upper bound log zeta}), we deduce that
\[
\sum_{\substack{n \leq y^K \\ P(n) \leq y \\ \omega(n) \geq K}} \frac{1}{n^{\alpha}} \ll \zeta (\alpha, y) \exp \bigl( -\delta^2 u / 4 \bigr).
\]
Recalling that $\delta = \frac{K}{u}-1$, this proves~\eqref{equ: divisors large deviation}, thus completing the proof of Lemma~\ref{l: good generating function}.
\end{proof}

\subsection{Proof of Lemma~\ref{l: large value bound}}
We now move on to the proof of Lemma~\ref{l: large value bound}. The proof follows quite similarly to bounds for moments of the Riemann zeta function~\cite{SoundMoments}.
\begin{proof}[Proof of Lemma~\ref{l: large value bound}]
We want to show that, uniformly for $k \leq \lfloor \frac{\log q}{\log y} \rfloor$ and for $(\log x)^4 \leq y \leq x^{\frac{1}{4 \log \log x}}$, there exists an absolute constant $c \geq 1$ so that for $V \geq c \sqrt{u}$, we have
\[
\p_{\chi(q)} \bigl( |P (\alpha / 2 + it, \chi, y) | > V \bigr) \ll \sqrt{k} \biggl( \frac{k u}{2 V^2} \biggr)^{k} ,
\]
uniformly for $t \in \R$, where we recall from~\eqref{equ: defn of P} that
\[
P (\alpha/2+it, \chi, y) = \sum_{j=1}^{\lfloor \frac{\log y}{\log 2} \rfloor} \frac{1}{j} \sum_{p \leq y^{1/j}} \frac{\chi(p)^j}{p^{j (\alpha/2+it)}}.
\]
The $j=1$ term should give the dominant contribution to the size of $P (\alpha / 2 + it, \chi, y)$, so we first remove the contribution from $j \geq 3$. Note that, by~\eqref{equ: saddle point est}, for $y \geq (\log x)^4$, we have the uniform bound $\alpha (x,y) \geq 3/4 (1+o(1))$. Hence
\[
\sum_{j=3}^{\lfloor \frac{\log y}{\log 2} \rfloor} \frac{1}{j} \sum_{p \leq y^{1/j}} \frac{\chi(p)^j}{p^{j (\alpha/2+it)}} \leq \sum_{j=3}^{\lfloor \frac{\log y}{\log 2} \rfloor} \frac{1}{j} \biggl( \frac{1}{2^{\alpha j / 2}} + \int_{2}^{\infty} \frac{d v}{v^{\alpha j / 2}} \biggr) \ll \sum_{j=3}^{\lfloor \frac{\log y}{\log 2} \rfloor} \frac{1}{2^{\alpha j / 2}} \leq A ,
\]
uniformly for $(\log x)^4 \leq y \leq x$, for some absolute constant $A > 0$. Therefore, by the union bound, we have
\begin{multline}\label{equ: union bound}
\p_{\chi(q)} \bigl( |P (\alpha / 2 + it, \chi, y) | > V \bigr) \leq \p_{\chi(q)} \biggl( \Bigl| \sum_{p \leq y} \frac{\chi (p)}{p^{\alpha / 2 + it}} \Bigr| > V - 2 U \biggr) \\ 
+ \p_{\chi(q)}\biggl( \Bigl| \sum_{p \leq \sqrt{y}} \frac{\chi(p^2)}{p^{\alpha + 2it}} \Bigr| > U \biggr),
\end{multline}
for any $U \geq A$. We begin by bounding the first probability on the right-hand side. Suppose further that $U \leq V/100$, say, then since $V - 2U \geq e^{-0.05} V$, we have
\[
\p_{\chi(q)} \biggl( \Bigl| \sum_{p \leq y} \frac{\chi (p)}{p^{\alpha / 2 + it}} \Bigr| > V - 2 U \biggr) \leq \p_{\chi(q)} \biggl( \Bigl| \sum_{p \leq y} \frac{\chi (p)}{p^{\alpha / 2 + it}} \Bigr| > e^{-0.05} V \biggr)
. \]
We now apply Markov's inequality and obtain
\[
\p_{\chi(q)} \biggl( \Bigl| \sum_{p \leq y} \frac{\chi (p)}{p^{\alpha / 2 + it}} \Bigr| > e^{-0.05} V \biggr) \leq \frac{1}{(e^{-0.05} V)^{2k}}\E_{\chi (q)} \Bigl| \sum_{p \leq y} \frac{\chi (p)}{p^{\alpha / 2 + it}} \Bigr|^{2k} .
\]
So we wish to obtain a good upper bound for 
\[
\E_{\chi (q)} \Bigl| \sum_{p \leq y} \frac{\chi (p)}{p^{\alpha / 2 + it}} \Bigr|^{2k}.
\]
These moments will be roughly Gaussian whenever $y^k \leq q$, and the handling of this term will follow identically to~\citet[Lemma~3]{SoundMoments}; we include the argument for the sake of completeness. As is done there, we begin by noting that 
\[
\biggl( \sum_{p \leq y} \frac{\chi (p)}{p^{\alpha / 2 + it}} \biggr)^{k} = \sum_{n \leq y^k} \frac{b_{k,y} \chi(n)}{n^{\alpha/2 + it}},
\]
where $b_{k,y} (n) = 0$ unless $n = p_1^{\beta_1} \dots p_r^{\beta_r}$ where $p_i \leq y$ and $\sum_{i=1}^r \beta_i = k$ (i.e.\,unless $n$ is a product of exactly $k$ primes less than or equal to $y$, counting multiplicity), in which case we have $b_{k,y} (n) = \frac{k!}{\beta_1! \dots \beta_r!} = {{k}\choose{\beta_1 , \dots, \beta_r}}$. Therefore, so long as $y^k \leq q$, by orthogonality, we have
\[
\E_{\chi (q)} \Bigl| \sum_{p \leq y} \frac{\chi (p)}{p^{\alpha / 2 + it}} \Bigr|^{2k} \leq \sum_{n \leq y^k} \frac{|b_{k,n}|^2}{n^{\alpha}} = \sum_{p_1 < \dots < p_r \leq x} \sum_{\substack{\beta_1 , \dots, \beta_r \geq 1 \\ \sum_{i} \beta_i = k}} {{k}\choose{\beta_1, \dots, \beta_r}}^2 \frac{1}{p_1^{\alpha \beta_1} \dots p_r^{\alpha \beta_r}}.
\]
Using the fact that ${{k}\choose{\beta_1 , \dots, \beta_r }} \leq k!$, this can be bounded above by
\[
k! \sum_{p_1 < \dots < p_r \leq y} \sum_{\substack{\beta_1, \dots, \beta_r \geq 1 \\ \sum_i \beta_i = k}} {{k}\choose{\beta_1, \dots, \beta_r}} \frac{1}{p_1^{\alpha \beta_1} \dots p_r^{\alpha \beta_r}} 
 = k! \biggl( \sum_{p \leq y} \frac{1}{p^\alpha} \biggr)^k .
\]
Then, using Stirling's formula to note that $k! \ll \sqrt{k} (k/e)^k$ (see~\cite[Equation~(B.26)]{MV2007}) and inserting these bounds back into the probability estimate, we have
\[
\p_{\chi(q)} \biggl( \Bigl| \sum_{p \leq y} \frac{\chi (p)}{p^{\alpha / 2 + it}} \Bigr| > V - 2 U \biggr) \ll \frac{ \sqrt{k} \bigl( \frac{k}{e} \sum_{p \leq y} \frac{1}{p^\alpha} \bigr)^k}{( e^{-0.05} V )^{2k}} = \sqrt{k} \biggl( \frac{k \sum_{p \leq y} \frac{1}{p^\alpha} }{e^{0.9} V^2} \biggr)^{k},
\]
uniformly for $A \leq U \leq V/100$. Note that, by equation~\eqref{equ: prime sum}, we can verify that for $y \leq x^{\frac{1}{4 \log \log x}}$, when $x$ is sufficiently large, we have $ \frac{1}{e^{0.9}} \sum_{p \leq y} \frac{1}{p^{\alpha}} \leq u/2 $. Combining this with the previous display gives
\[
\p_{\chi(q)} \biggl( \Bigl| \sum_{p \leq y} \frac{\chi (p)}{p^{\alpha / 2 + it}} \Bigr| > V - 2 U \biggr) \ll \sqrt{k} \biggl( \frac{ku}{2V^2} \biggr)^k .
\]
Therefore, the first term in~\eqref{equ: union bound} is consistent with the bound in the lemma. 

To complete the proof, we need to prove a similar bound for the second term in~\eqref{equ: union bound}. Specifically. we will show that for some $A \leq U \leq V/100$, we have
\[
\p_{\chi(q)}\biggl( \Bigl| \sum_{p \leq \sqrt{y}} \frac{\chi(p^2)}{p^{\alpha + 2it}} \Bigr| > U \biggr) \ll \sqrt{k} \biggl( \frac{k u}{e V^2} \biggr)^k .
\]
We begin similarly. In this case, we have
\[
\biggl( \sum_{p \leq \sqrt{y}} \frac{\chi(p)^2}{p^{\alpha + 2it}} \biggr)^k = \sum_{n \leq y^{k/2}} \frac{b_{k, \sqrt{y}} (n)}{n^{\alpha + 2it}},
\]
so under the assumption $y^k \leq q$, by a similar argument, we have
\[
\E_{\chi(q)} \Bigl| \sum_{p \leq \sqrt{y}} \frac{\chi(p^2)}{p^{\alpha + 2it}} \Bigr|^{2k} \leq \sum_{n \leq y^{k/2}} \frac{|b_{k, \sqrt{y}} (n)|^2}{n^{2\alpha}} \leq k! \biggl( \sum_{p \leq \sqrt{y}} \frac{1}{p^{2 \alpha}} \biggr)^k \ll B^k k! ,
\]
for some absolute constant $B>0$, where we have again used the fact that $\alpha (x,y) \geq 3/4 (1+o(1))$. By Markov's inequality and Stirling's formula, we have
\[
\p_{\chi(q)}\biggl( \Bigl| \sum_{p \leq \sqrt{y}} \frac{\chi(p^2)}{p^{\alpha + 2it}} \Bigr| > U \biggr) \ll \frac{B^k k!}{U^{2k}} \ll \sqrt{k} \biggl( \frac{k B}{e U^2} \biggr)^{k} .
\]
The choice $U = V \sqrt{B/u}$ delivers the desired bound, which is admissible so long as $A / V \leq \sqrt{B/u} \leq 1/100$. These are satisfied when $x$ is large so long as $V \geq c \sqrt{u}$ for some absolute constant $c \geq 1$ depending only on $A$ and $B$, which gives the constant seen in Lemma~\ref{l: large value bound}.
\end{proof}

\section{Generalisations}\label{s: other families}

\subsection{Proofs of Theorems~\ref{t: mult twists} and~\ref{t: other families}}
The proof of Theorem~\ref{t: main} is amenable to generalisation in a straightforward manner. Suppose that $N \geq 1$ is some parameter (representing the size of a family) and that $\mathcal{F}_N$ is a finite or infinite collection of completely multiplicative functions $\psi \colon \N \rightarrow \C$ that satisfy $|\psi(n)| \leq 1$ for all $n \in \N$. Suppose further that we can put a probability measure $\p_N$ on $\mathcal{F}_N$ such that, for any complex coefficients $(c(n))_{n \in \N}$, the induced expectation $\E_N$ satisfies
\begin{equation}\label{equ: generalisation cond}
\E_{N} \Bigl| \sum_{n \leq X} c(n) \psi(n) \Bigr|^2 \ll \sum_{n \leq X} |c(n)|^2 ,
\end{equation}
whenever $X \leq N$, where the implied constant is independent of the size of the collection, $N$. We then claim that for $2 \leq x \leq N$ and $(\log x)^6 \leq y \leq x^{\frac{1}{32 \log \log x}}$, we have
\begin{equation}\label{eqn: claim}
\E_{N} \Bigl| \sum_{\substack{n \leq x \\ P(n) \leq y}} \psi (n) \Bigr| \ll \sqrt{\Psi(x,y)S(x,y,N)}, 
\end{equation}
as in Theorem~\ref{t: main}. We will now show how this claim implies Theorems~\ref{t: mult twists} and~\ref{t: other families}, for which we need to show that the corresponding families satisfy~\eqref{equ: generalisation cond}. Take any completely multiplicative function $h \colon \N \rightarrow \C$ satisfying $|h(p)| \leq 1$ on all primes $p$, and consider the infinite collection\footnote{We use the variable $v$ as opposed to $t$ here, since $t$ appears throughout the paper as the imaginary component that comes from Perron's formula, as in~\eqref{equ: Lemma 1 for multiplicative functions}.} of completely multiplicative functions $\mathcal{F}_{T} = \{ n \rightarrow h(n) n^{iv}: v \in [T,2T]\}$. We associate with $\mathcal{F}_T$ the probability measure induced by pulling back to $v \in [T,2T]$ and sampling $v$ uniformly at random, and we call the expectation $\E_T$. For $X \leq T$, and for arbitrary complex coefficients $(c(n))_{n \in \N}$, the classical mean value theorem~\cite[Theorem~9.1]{IK} then tells us that
\[
\E_T \Bigl| \sum_{n \leq X} c(n) h(n) n^{iv} \Bigr|^2 =  \frac{1}{T} \int_{T}^{2T} \Bigl| \sum_{n \leq X} c(n) h(n) n^{iv} \Bigr|^2 \, dv \ll \sum_{n \leq X} |c(n) h(n)|^2 \ll \sum_{n \leq X} |c(n)|^2 ,
\]
whenever $X \leq T$, hence~\eqref{equ: generalisation cond} holds in this case. Again fixing $h$, we can also consider $\mathcal{F}_{Q^2} = \{ n \rightarrow h(n) \chi(n): \chi \text{ is primitive} \bmod q \text{ for some } q \leq Q \}$ with the uniform (discrete) probability measure. Note that this set of functions has $\asymp Q^2$ many elements. We call the corresponding expectation $\E_{Q^2}$. Similarly to above, for any $X \leq Q^2$, by the multiplicative large sieve inequality~\cite[Theorem~7.13]{IK}, we have 
\[
\E_{Q^2} \Bigl| \sum_{n \leq X} c(n) h(n) \chi(n) \Bigr|^2 \ll \frac{1}{Q^2} \sum_{q \leq Q} \, \, \sideset{}{^*} \sum_{\chi \bmod q} \Bigl| \sum_{n \leq X} c(n) h(n) \chi(n) \Bigr|^2 \ll \sum_{n \leq X} |c(n) h(n)|^2 ,
\]
which is $\ll \sum_{n \leq X} |c(n)|^2$, where $\sum\nolimits^{*}_{\chi \bmod q}$ denotes the sum over primitive Dirichlet characters modulo $q$. Therefore~\eqref{equ: generalisation cond} is also satisfied (with $N = Q^2$) in this case.

The proof of the claim \eqref{eqn: claim} follows by replacing occurrences of $\chi (n)$ with $\psi (n)$ and performing analogous calculations: we will briefly discuss how one proves analogues of Lemmas~\ref{l: good generating function} and~\ref{l: large value bound}, the main result then follows by exchanging occurrences of $\chi$ with $\psi$ in Section~\ref{s: Proof of main theorem}. First, as in~\eqref{equ: defn of P}, we consider
\[
P (s, \psi, y) = \sum_{j = 1}^{\lfloor \frac{\log y}{\log 2} \rfloor} \frac{1}{j} \sum_{p \leq y^{1/j}} \frac{\psi(p)^j}{p^{js}}.
\]
To prove an analogous result to Lemma~\ref{l: good generating function}, we need to bound
\begin{equation}\label{equ: Lemma 1 for multiplicative functions}
\E_{N} \Bigl| \sum_{\substack{n \leq y^K \\ P(n) \leq y}} \frac{\psi(n)}{n^{\alpha / 2 + it}} - \sum_{k = 0}^{K} \frac{P (\alpha / 2 + it, \psi, y)^{k}}{k!} \Bigr|.
\end{equation}
Mirroring~\eqref{equ: truncated expansion}, for any $K \in \N$, we have
\[
\sum_{k=0}^{K} \frac{P (\alpha/2 + it, \psi, y)^k}{k!} = \sum_{k=0}^{K} \frac{1}{k!} \biggl( \sum_{j = 1}^{\lfloor \frac{\log y}{\log 2} \rfloor} \frac{1}{j} \sum_{p \leq y^{1/j}} \frac{\psi (p)^j}{p^{j(\alpha / 2 + it)}} \biggr)^k = \sum_{\substack{n \leq y^K \\ P(n) \leq y}} \frac{a(n) \psi (n)}{n^{\alpha/2+it}},
\]
and as before, the coefficients $a(n)$ satisfy the conditions from~\eqref{equ: a(n) properties}, specifically
\[
a(n) = 1 \text{ if } p^j \mid n \Rightarrow p \leq y^{1/j} \text{ \emph{and} } \omega(n) \leq K. \text{ Otherwise, } 0 \leq a(n) \leq 1.
\]
In fact, these coefficients are exactly the same as those appearing in~\eqref{equ: truncated expansion}. Therefore, by considering the difference in~\eqref{equ: Lemma 1 for multiplicative functions}, applying Cauchy--Schwarz followed by~\eqref{equ: generalisation cond}, we see that we can obtain the same bound as in Lemma~\ref{l: good generating function} for this difference. The proof of Lemma~\ref{l: large value bound} with $P (\alpha/2+it, \chi, y)$ replaced by $P (\alpha/2+it, \psi, y)$ also follows identically, the only inputs being orthogonality and the fact that $|\psi (n)| \leq 1$ for all $n \in \N$.

\subsubsection*{Rights Retention}
For the purpose of open access, the authors have applied a Creative Commons Attribution (CC-BY) licence to any Author Accepted Manuscript version arising from this submission.

\printbibliography
\end{document}